\documentclass{amsart}
\usepackage{amsthm}
\usepackage{amssymb,amsmath,mathrsfs,amsthm,xspace}
\usepackage[all,v2,cmtip,2cell]{xy}
\usepackage{mathrsfs}
\usepackage[OT1,T1]{fontenc}

\UseAllTwocells
\newcounter{zlist}

\newcounter{blist}

\newcounter{rlist}

\swapnumbers
\newtheorem{theorem}{Theorem}[section]
\newtheorem{lemma}[theorem]{Lemma}
\newtheorem{thm}[theorem]{}
\newtheorem{proposition}[theorem]{Proposition}

\newtheorem{remark}[theorem]{Remark}

\numberwithin{equation}{section}

\input xy
\xyoption{all}
\xyoption{2cell}
\UseAllTwocells

\newcommand{\A}{\mathscr{A}}
\newcommand{\B}{\mathscr{B}}

\newcommand{\X}{\mathscr{X}}

\newcommand{\E}{\mathcal{E}}
\newcommand{\mP}{\mathcal{P}}
\newcommand{\W}{\textsf{Mod}}
\newcommand{\ot}{\otimes}
\newcommand{\CC}{{\mathcal {S}l}}

\newcommand{\Hom}{\textsf{Hom}}
\newcommand{\Set}{\textsf{Set}}
\newcommand{\Mat}{\textsf{Mat}}

\newcommand{\xr}{\xrightarrow}

\newcommand{\ve}{\varepsilon}
\newcommand{\T}{\ensuremath{\mathscr{T}}}
\newcommand{\R}{\ensuremath{\mathscr{S}}}
\newcommand{\V}{{\mathscr{V}}}
\newcommand{\mQ}{\mathscr{Q}}

\makeatletter
\@namedef{subjclassname@2020}{%
  \textup{2020} Mathematics Subject Classification}
\makeatother
\begin{document}

\title{Morita theory for quantales}

\author{Bachuki Mesablishvili}
\address{A. Razmadze Mathematical Institute and Department of Mathematics, Faculty of Exact and Natural Sciences
of I. Javakhishvili Tbilisi State University, Tbilisi, Georgia}
\email{bachuki.mesablishvili@tsu.ge}
\thanks{This work was supported by Shota Rustaveli National Science Foundation of Georgia (SRNSFG) (Grant   FR-22-4923).}

\begin{abstract} Morita theory for quantales is developed. The main result of the paper is a characterization of those quantaloids 
(categories enriched in the symmetric monoidal closed category of sup-lattices) that are equivalent to modular categories over quantales. 
Based on this characterization, necessary and sufficient conditions are derived for two quantales to be Morita-equivalent, i. e. have 
equivalent module categories. As an application, it is shown that the category of internal sup-lattices in a Grothendieck topos is 
equivalent to the module category over a suitable chosen ordinary quantale. 
\end{abstract}

\smallskip  

\keywords{Monad, adjoint triangle, quantale, modules over a quantale, Morita equivalence, (localic) Grothendieck topos.}

\subjclass[2020]{18B25; 18C15; 18C20; 18D20; 18E08; 18F10; 18F75; 18M05}

\maketitle


\section{Introduction}
The goal of this paper is to develop Morita theory for quantales. Our starting point for doing this is the so-called 
\emph{comparison theorem} (see \cite{Mes0, MW, MW1}). Starting from a monad $\T$ on a category $\X$ and an adjunction 
$F \dashv U :\A \to \X$, there is a bijective correspondence between functors $K:\A \to \X^\T$ with $U=U^\T K$
and morphisms of monads $\T \to \R$, where $\X^\T$ is the Eilenberg-Moore category of $\T$-algebras, $U^\T: \X^\T \to \X$
is the evident forgetful functor and $\R$ is the monad on $\X$ generated by the adjunction $F \dashv U$ (\cite{Du}). 
When the monad morphism is an isomorphism, the functor $U$ is called \emph{$\T$-Galois} in \cite{MW0}. Then the comparison theorem
(see Theorem \ref{equivalence.}) asserts that the functor $K$ is an equivalence of categories if
and only if $U$ is monadic and $\T$-Galois. Using the comparison theorem, we prove our central result
(Theorem \ref{sup.l.th.1}) characterizing module categories over quantales among quantaloids. The key property of module
categories over quantales, which distinguishes them among quantaloids, is that they are monadic over the category of sets, 
$\Set$. With help of this central result, we
then transfer many aspects of the classical Morita theory to the context of quantales. In particular,
we describe Morita equivalence of quantales both in terms of projective generators (Theorem \ref{morita.2.}) and full idempotents
(Theorem \ref{m.full.idemp.}). Note that Theorem \ref{m.full.idemp.} essentially describes  all quantales that are Morita equivalent to a 
given one.

The outline of this paper is as follows. After recalling in Section 2 those notions and results that will be used from category 
theory, we state the comparison theorem, which gives  a necessary and sufficient condition for a functor into the Eilenberg-Moore 
algebras for a monad to be an equivalence.  In section 3 we briefly review some of background knowledge about quantales and  modules over them.
In Section 4 we present our main result providing conditions  under which a quantaloid can be recognized as a module category 
over a quantale. In the next two sections, the preceding results are applied  to derive necessary and sufficient
conditions for two quantales to be Morita equivalent both in terms of projective generators (Section 5) and in terms of idempotents 
and matrices (Section 6). Finally, in the concluding section, some applications to internal sup-lattices in Grothendieck toposes 
are considered. In particular, the category of internal sup-lattices in a Grothendieck topos is characterized as the category
of  external modules over an external sup-lattice. This allows us to recover a result of Joyal and Tierney \cite[Chapter VI, Proposition 3.1]{JT} on 
internal locales in localic Grothendieck toposes.

We refer to \cite{Mcl}, \cite{Bo2} and \cite{Sch} for basic category theoretic notions and terminology.

\section{Categorical preliminaries}
For a monad $\T=(T, \mu, \eta)$ on a category $\X$, we write
\begin{itemize}
\item[--] $\X^{\T}$ for the Eilenberg--Moore category of $\T$--algebras;
\item[--] $U^{\T} : \X^{\T} \to \X ,\,\, (X, h) \to X,$ for the
underlying (forgetful) functor;
\item[--] $F^{\T} : \X \to \X^{\T},\,\, X \to (T(X),\mu_X),$ for the free
$\T$--algebra functor, and
\item[--] $\eta^{\T}, \ve^{\T}: F^{\T}\dashv U^{\T} : \X^{\T} \to \X$ for the forgetful--free adjunction.
(Recall that $\eta^{\T}=\eta$ and $(\ve^{\T})_{(X,h)}=h$  for all $(X,h)\in \X^{\T}$.)
\end{itemize}

\noindent Let $ \eta , \ve \colon F \dashv U \colon \A \to \X$ be an adjunction, $\R=(UF, U \ve F,\eta)$ be
the monad on $\X$ generated  by the adjunction and $K^\R : \A \to \X^{\R}$ be the comparison functor.
Recall that $K^\R$ assigns to each object $A \in \A$ the $\T$-algebra $ (U(A), U(\ve_A)),$ and to each morphism
$f:A\to A'$ the morphism $U(f): U(A) \to U(A').$ Moreover, $U^{\T} K^\T = U$ and $K^\T F = F^{\T}$.
One says that the functor $U$ is \emph{monadic} if $K^\T$ is an equivalence of categories.

\begin{thm} \label{comparison.th.}{\bf The comparison theorem.}  \em Suppose now that $\eta, \ve : F \dashv U: 
\A \to \X$ is an adjunction, $\T$ is a
monad on $\X$ and $K:\A \to \X^\T$ is a functor
such that $U^\T K=U$. The situation may be pictured as
\begin{equation}\label{adj.}
\xymatrix
@R=28pt@C=40pt
{ \A \ar@<.6ex> [dr]^U \ar[rr]^{K} && \X^\T \ar@<-.6ex> [dl]_{U^\T}\\
& \X \ar@<.6ex> [ul]^F \ar@<-.6ex> [ru]_{F^\T}\,& }
\end{equation} Such diagrams are called \emph{adjoint triangles} in \cite{Du}.
Write $\R$ for the monad on $\X$ generated by the adjunction $F\dashv U$.  It is shown in \cite{Du}
that $s_K=U^\T \gamma_K$, where $\gamma_K$ is the composite
\[ F^\T \xr{ F^\T \eta}F^\T UF= F^\T U^\T KF \xr{ \ve^\T KF} KF,\]
is a monad morphism $\T \to \R$. When $s_K$  (or, equivalently, $\gamma_K$) is a natural
isomorphism, one says that the functor $U$ is $\T$-\emph{Galois} (see \cite[Definition 1.3]{MW}).

\vskip.2in

Applying the dual of \cite[Theorem 4.4]{Mes0} (or  \cite[Proposition 2.1]{MW}) to the data given in the diagram (\ref{adj.}), we get
the following \emph{comparison theorem}:

\begin{theorem} \label{equivalence.} In the situation described above, $K$ is an equivalence of categories
if and only if
 \begin{itemize}
   \item [(i)] $U$ is monadic and
   \item [(ii)]  the induced monad morphism $s_K : \T \to \R$ is an isomorphism (\,$U$ is $\T$-Galois).
 \end{itemize}
 \end{theorem}
 
\end{thm}

\begin{thm} \label{Cauchy}{\bf Splitting of idempotents.}  \em Given a monoid $(M, \ast, \textsf{1})$,
an element $m \in M$ is an \emph{idempotent} if $m \ast m=m$. A morphism $e:A \to A$ in a category
is idempotent if $e$ is an idempotent element in the endomorphism monoid $\A(A,A)$. Thus, $e$ is idempotent if $e^2=e\circ e = e$. 
A category $\A$ is said to be \emph{Cauchy complete} if \emph{idempotents split} in
$\A$, i.e., for every idempotent $e:A \to A$, there exists an object $A_e$ (called the \emph{splitting object})
and morphisms $p : A \to A_e$ and $i:A_e \to A$ such that $ip = e$ and $pi = 1_{A_e}$\,. (Note that such a splitting,
and hence also the splitting object, is unique up to isomorphism.)   In this case, $(A_e, i)$
is an equaliser of $e$ and $1_A$ and $(A_e, p)$ is a coequaliser of $e$ and $1_A$. Thus,  any category
admitting either equalisers or coequalizers is Cauchy complete. 
Note that if $\T$ is a monad on a Cauchy complete
category $\A$, then so is $\A^\T$. Moreover, the forgetful functor $U^\T: \A^\T \to \A$
preserves and creates splitting of idempotents. 
\end{thm}

\begin{thm} \label{generator.def.}{\bf Generators.}  \em Recall that a family of functors $(F_i:\B \to \B_i)_{i \in I}$
is called \emph{jointly faithful} (resp. \emph{conservative}) if the associated functor \[\B \to \prod_{i \in I}\B_i,\,\,
B \longmapsto (F_i(B))_{i \in I}\] to the product category is faithful (resp. conservative); more explicitly, a family of functors
$(F_i:\B \to \B_i)_{i \in I}$ is jointly faithful if for each pair of distinct morphisms $f,g :B \rightrightarrows B'$ in
$\B$ one has $F_i(f)\neq F_i(g)$ for some $i \in I$; and is jointly conservative if any morphism $f:B \to B'$ in $\B$
such that $F_i(f)$ is an isomorphism in $\B_i$ for all $i \in I$, is an isomorphism in $\B$. If $\B$ has equalizers, a jointly 
conservative family of functors is always jointly faithful; the converse holds, for example, if every morphism in $\B$ which 
is both a monomorphism and an epimorphism is already an isomorphism (that is, if $\B$ is \emph{balanced}), but not in general.  
A set of objects $\{B_i, i \in I\}$ of a category $\B$ is a \emph{(strong) generating set} for $\B$ if the family of functors 
$(\B(B_i,-): \B \to \Set)_{i \in I}$ jointly faithful (conservative). When the set $\{B_i, i \in I\}$ is reduced to a single 
element $\{B\}$, one says that $B$ is a \emph{(strong) generator} for $\B$.

Recall further (for example, from \cite{BT}) that for a right adjunction functor is 
\begin{itemize}
  \item [--] faithful if and only if all the components of the counit of the adjunction is epimorphic;
  \item [--] faithful and conservative  if and only if all the components of the counit of the 
  adjunction is an extremal epimorphism; that is, an epimorphism that factorizes through no proper subobject
  of its codomain.
\end{itemize}

\noindent In particular, given an object $B$ of a category $\B$, if $\B$ admits arbitrary small copowers of $B$ (and hence the functor $B^{(-)}$
that takes a set $X$ to the coproduct of $X$ copies of $B$, is a left adjoint to $\B(B,-):\B \to \Set$),
then $B$ is a (strong) generator if and only if for any object $B' \in \B$, the canonical morphism $B^{\B(B,B')}\to B'$, which is
the $B'$-component of the counit of the adjunction $\B(B,-) \dashv B^{(-)}$, is an (extremal) epimorphism. If each of these components
is a regular epimorphisms (i.e., a coequalizer of a pair of morphisms), then the functor $\B(B,-)$ has a full and faithful comparison functor 
into its Eilenberg-Moore category. Such an object $B$ is called \emph{regular generator} for $\B$.
Evidently, if regular epimorphisms coincide with extremal epimorphisms (which is the case, for example, for \emph{regular}
categories in the sense of \cite{Br}),  a strong generator is a regular generator.
Note finally that if a (regular, strong) generator $B$ is a retracts of an object $B'$, then $B'$ is also a (regular, strong) generator.
This follows from the fact that in any category, a retracts of a (regular, extremal) epimorphism is again a (regular, extremal) epimorphism.
 \end{thm}

\begin{thm} \label{projective.}{\bf Projective objects.}  \em An object $B$ of a category $\B$ is called
\emph{(regular) projective} if the hom functor $\B(B,-):\B \to \Set$ takes (regular) epimorphisms to epimorphisms.
In any category, every retract of a (regular) projective object is (regular) projective and a coproduct of (regular) 
projective objects is  (regular) projective (provided this coproduct exists). An object is called a \emph{(regular) 
projective generator}  if it is simultaneously a (regular) generator and (regular) projective.
\end{thm}

\
\
\

\section{Quantales and their modules}

In this section we provide basic notions and properties of quantales and modules over them.

\begin{thm}\label{algebras}{\bf Quantales.} \em Let $\CC$ be the category of sup-lattices (=complete posets) and 
sup-preserving maps. The top element of a sup-lattice $L$ is denoted by $1$, and the bottom by $0$. It is well known 
(see, \cite{JT}) that $\CC$ is a symmetric monoidal closed category. We recall that the tensor product
$L\ot L'$ of two sup-lattices $L$ and $L'$ is the codomain of the universal map $L \times L' \to L\ot L'$ that preserves
supremum in each variable separately (such a map is called a \emph{bimorphism}). The tensor unit is $\textsf{2}=\{0\leq 1\}$.
The internal hom is the set $[L,L'\,]$ of sup-preserving maps ordered by $f\leq g$ if and only if $f(l)\leq g(l)$
for all $l \in L$. Since the order on $L'$ is complete, so is the pointwise order on $[L,L']$. Recall (e.g., from \cite{JT}) 
that the category $\CC$ of sup-lattices is monadic over $\Set$ via the functor $\CC(\textsf{2}, -):\CC \to \Set$.
Thus, $\CC$ is a complete, cocomplete and Barr-exact category (in the sense of \cite{Br}) and hence in particular idempotents
split in $\CC$. Moreover, (small) coproducts are biproducts; that is, if $\{ X_{i}\}$ is a (small) family of objects of $\CC$, 
the comparison  morphism $\coprod {X_i} \to \prod {X_i}$ is an isomorphism.

A \emph{quantale} is a monoid in the symmetric monoidal category $\CC$.
More precisely, a quantale is a sup-lattice $A$  equipped with a monoid structure  $(A,\ast, 1)$
such that $\ast$ distributes over suprema:
\begin{equation}\label{quantale.}
a \ast (\bigvee_{i \in I}b_i)=\bigvee_{i \in I}(a \ast b_i) \,\,\,\text{and} \,\,\,
(\bigvee_{i \in I}a_i)\ast b=\bigvee_{i \in I}(a_i\ast b)\end{equation} for all
$a,b \in A$ and all $\{a_i\}_{i \in I}, \{b_i\}_{i \in I}\subseteq A$. A quantale $A$ is \emph{commutative}
if $a \ast b=b\ast a$ for all $a,b \in A$ and is \emph{idempotent} if $a \ast a=a$ for all $a\in A$.

A morphism $f:A \to B$ of quantales is a morphism of sup-lattices satisfying
\begin{itemize}
  \item [(i)]$f(1)=1$
  \item [(ii)]$f(a \ast a')=f(a) \ast f(a')$
\end{itemize} for all $a,a' \in A$.
\end{thm}

\begin{thm} \label{modules.}{\bf Modules over quantales.}  \em
With any quantale $(A,\ast, 1)$, there is associated a category of (right) modules over it.
Recall that a right $A$-module is a complete lattice $M$ together with a map
$-\cdot-:M \times A \to M$ such that

\begin{equation} \label{module}
    \begin{aligned}
      \quad &m \cdot \textsf{1}=m \\
      \quad & m \cdot  (a_1\ast a_2)=(m \cdot a_1) \cdot a_2 \\ 
      \quad &(\bigvee_{i \in I}m_i) \cdot a=\bigvee_{i \in I}( m_i \cdot a)\\
      \quad & m\cdot (\bigvee_{i \in I}a_i)=\bigvee_{i \in I}(m \cdot a_i)
    \end{aligned}
\end{equation} for all $a, a_1,a_1 \in A$, $m \in M$, $\{a_i\}_{i \in I} \subseteq A$ 
and $\{m_i\}_{i \in I}\subseteq M$. Note that the third and fourth conditions are equivalent to saying that
the map $-\cdot-$ is a bimorphism. 

A morphism $f: M \to N$ of right $A$-modules is a morphism of sup-lattices such that $f(m\cdot a)=f(m) \cdot a$ for 
all $m\in M$ and $a \in A$. The category of right $A$-modules is denoted by $\W_A(\CC)$. The evident forgetful functor 
$U^A :\W_A(\CC)\to \CC$ is monadic with left adjoint $F^A$, which takes  $X \in \CC$ to $X \ot A$, and hence it creates 
all limits. Moreover, since the functor $-\ot A : \CC \to \CC$ admits as a right adjoint the functor $[A,-] : \CC \to \CC$,
the forgetful functor $U^A$ has a right adjoint sending $X \in \CC$ to $[A,X]$. Then $U^A$ clearly creates all colimits.
In a similar way, one has the category of left $A$-modules $_A\W(\CC)$. If $B$ is another quantale, a $(B,A)$-bimodule
is a sup-lattice together with compatible left and right actions of $B$ and  $A$, respectively. The corresponding
category is denoted by $_{B}\W_A(\CC)$.

In the sequel, given a quantale $(A,\ast, \textsf{1})$, we usually omit the symbol $\ast$ in notation and simply write $aa'$
instead of $a\ast a'$. 
\end{thm}

\begin{thm} \label{quantaloids.}{\bf Quantaloids.}  \em A \emph{quantaloid} is a locally small category enriched in the 
symmetric monoidal closed category $\CC$. If $\mQ$ is a quantaliod, then one has the following commutative diagram of 
categories and functors:
$$\xymatrix @C=.6in @R=.5in{& \CC\ar[d]^{\CC(\textsf{2},-)}\\
\mQ^{\text{op}}\times \mQ \ar[ru]^{\Hom_\mQ(-,-)} \ar[r]_-{\mQ(-,-)}&\Set,}$$
in which $\Hom_\mQ(-,-)$ is the so called \emph{lifted hom-functor}.
The property that (small) coproducts are biproducts in $\CC$,
transfers to any quantaloid $\mQ$, (e.g., \cite{HST}), in the sense that for any family $(D_i, i \in I)$ of objects of $\mQ$
the product $\prod_{i \in I} D_i$ exists if and only if the coproduct $\coprod _{i \in I} D_i$ exists and, moreover,
the canonical comparison morphism between them is
an isomorphism. We will write $\oplus {D_i}$ for the biproduct of the family $(D_i, i \in I)$.
Typical examples of quantaloids are module categories over quantales. Given a quantale $A$, we shall usually write $\Hom_A(-,-)$
instead of $\Hom_{\W_A(\CC)}(-,-)$. Recall that for $M,N \in \W_A(\CC)$,  $\W_A(M,N)$ is a sup-lattice with pointwise  
defined operations.

\end{thm}

We close this subsection with the following important observation.
\begin{remark}\label{remark.proj.} \em Since the forgetful functor $\CC(\textsf{2},-):\CC\to \Set$ takes epimorphisms to 
split epimorphisms \cite{JT}, each epimorphism is regular in $\CC$ (e.g.,\cite{KP}), implying that regular epimorphisms, 
strong epimorphisms and epimorphisms coincide in $\CC$. Since for any quantale $A$, the module category $\W_A(\CC)$ is 
both monadic and comonadic over $\CC$, the same hold also in $\W_A(\CC)$. Therefore, projective objects in module categories 
over quantales coincide with the regular projective objects (see, Subsection \ref{generator.def.}). Moreover, generators, 
strong generators, and regular generators all coincide.
\end{remark}

\bigskip
\bigskip

\section{When is a quantaloid equivalent to a module category over a quantale ?}

In this section, we characterize those quantaloids that are equivalent to module categories over quantales.

\medskip

Let $\A$ be an arbitrary category and  $A\in \A$ be an object of $\A$. For any set $X$, we write $A^{(X)}$
(resp. $A^{X}$) for the \emph{copower} (resp. \emph{power}) of $A$ by $X$, i.e., the coproduct $\coprod_{x \in X} A$ 
(resp. product $\prod_{x \in X} A$) in $\A$ of $X$ copies of $A$, provided they exist.  The coproduct injection $A \to A^{(X)}$ 
corresponding to $x \in X$ is denoted by $\iota^A_x$. If $f:X \to Y$
is a map of sets, then $A^{(f)} :A^{(X)}\to A^{(Y)}$ is the unique morphism in $\A$ such that $A^{(f)} \cdot \iota^A_x =
\iota_{f(x)}$ for all $x \in X$. Dually, one has the canonical projections $\pi_x:\prod_{x \in X}A \to A$  and the morphism
$A^f:A^Y \to A^X$.

\begin{thm} \label{quantaloids.tr.}{\bf Adjoint triangle for quantaloids.}  \em A
Let $\mQ$ be a quantaloid  and $Q$ an arbitrary but fixed  object  of $\mQ$.  To $Q$ there is  associated
a functor $\Hom_\mQ(Q,-):\mQ \to \CC$ and the quantale $E=\Hom_\mQ(Q,Q)$ of endomorphisms of $Q$. It is easy to check that
for any object $A\in \mQ$, the object $\Hom_\mQ(Q,A)$  naturally has the structure of a right $E$-module
given by the composition
\[\Hom_\mQ(Q,A)\otimes \Hom_\mQ(Q,Q)\xr{} \Hom_\mQ(Q,A)\]
and thus the assignment $A \longmapsto \Hom_\mQ(Q,A)$ yields a functor $K^Q:\mQ \to \W_E(\CC)$
such that $U^EK^Q=\Hom_\mQ(Q,-)$. When, as we henceforth suppose,  $\mQ$  admits set-indexed copowers of $Q$, 
the functor\[\mQ(Q,-)=\CC(\textsf{2},-)\circ \Hom_\mQ(Q,-):\mQ \to \Set\] admits as a left adjoint the functor 
\[X \mapsto Q^{(X)}, \,\, (f:X \to Y)\longmapsto  (Q^{(f)}:Q^{(X)}\to Q^{(Y)}).\] For any set $X$, the $X$-component
$(\eta^Q)_X : X \to \mQ(Q,Q^{(X)})$ of the unit $\eta^Q$ of the adjunction $Q^{(-)} \dashv \mQ(Q,-)$ is
the map that takes $x \in X$ to  $\iota^Q_x :Q \to Q^{(X)})$,
while the $A$-component $(\ve^Q)_A$, $A \in \mQ$, of the
counit $\ve^Q$ is the morphism $Q^{{\mQ}(Q,A)} \to A$ defined (uniquely) by the equality
$(\ve^Q)_A \circ \iota^Q_f=f.$

Next, composing adjunctions \[F^E \dashv U^E:\W_E(\CC) \to \CC \,\, \text{and}\,\,\textsf{2}^{(-)}
\dashv \CC(\textsf{2},-):\CC \to \Set\] gives an adjunction $E^{(-)} \dashv U_E : \W_E(\CC) \to \Set$, where
$U_E$ is the evident forgetful functor, while $E^{(-)}$ sends
a set $X$ to the free right $E$-module $E^{(X)}$, which because of the isomorphism $E^{(X)} \simeq E^{X}$, can be seen
as the set of all maps $X \to E$.  Since the forgetful functor $U^E:\W_E(\CC) \to \CC$ is monadic and admits both adjoints, 
its composite with the monadic $\CC(\textsf{2},-):\CC \to \Set$ is also monadic. Therefore, the forgetful functor 
$U_E:\W_E(\CC) \to \Set$ with left adjoint $E^{(-)}$ is monadic. Then $\W_E(\CC)$ is (isomorphic to) the category of 
$\T_E$-algebras, where $\T_E$ is the monad on $\Set$ generated by the adjunction $E^{(-)} \dashv U_E$. Consequently, 
since evidently $U_E \circ K^Q=\mQ(Q,-)$, we have the following  particular instance of the diagram (\ref{adj.}):
\begin{equation}\label{set.case.}
\xymatrix @R=20pt@C=40pt
{\mQ \ar@<.6ex> [ddr]^{\mQ(Q,-)} \ar[rr]^{K^Q} && \W_E(\CC) \ar@<-.6ex> [ddl]_{U_E}&\\
&&**[r]U_E \circ K^Q={\mQ}(Q,-) &\\
& \Set \ar@<.6ex> [uul]^{Q\,^{(-)}} \ar@<-.6ex> [ruu]_{E^{(-)}}\,&& }
\end{equation} Write $\eta^E$ and $\ve^E$ for the unit and counit of the adjunction $E^{(-)} \dashv U_E$.  Then for any $X \in \Set$,
$(\eta^E)_X : X \to E^{(X)}\simeq E^X$ is the map that takes $x \in X$ to the map supported at $x$ with value $\textsf{id}_Q$;
and for any $M\in \W_E(\CC)$,~$(\ve^E)_{M}: E^{(M)}\to M$ is the map $(e_m)_{m \in M} \to \bigvee_{m \in M} ~{m\cdot e_m}.$
\end{thm}

\begin{thm} \label{self-small.}{\bf Self-small objects.}  \em We intend to apply Theorem \ref{equivalence.}
to the above adjoint triangle, but before doing so, we must first examine those  objects $Q \in \mQ$ 
for which the functor $\mQ(Q,-)$ is $\T_E$-Galois.

Let $\T_Q$ denote the monad on $\Set$ generated by the adjunctions  $Q^{(-)}\dashv \mQ(Q,-)$. Using the description of
the unit of the adjunction $Q^{(-)}\dashv \mQ(Q,-)$ and counit of the adjunction $E^{(-)} \dashv U_E$, one easily checks 
that for any set $X$, the $X$-component $\gamma_X$ of the natural transformation $\gamma=\gamma_{K^Q}:E^{(-)} \to K^Q \circ Q^{(-)}$ 
for which $U_E\gamma$ is the induced monad morphism $s_Q=s_{K^Q} : \T_E \to \T_Q$, is the unique morphism
$$\Hom_\mQ (Q,Q)^{(X)} \to \Hom_\mQ(Q,Q^{(X)})$$ such that \[\gamma_X \cdot \iota^{\Hom_\mQ (Q,Q)}_{x}=\Hom_\mQ (Q,\iota_x^Q)\] 
and hence $\gamma_X$ is precisely the canonical comparison morphism \[\Hom_\mQ (Q,Q)^{(X)} \to \Hom_\mQ(Q,Q^{(X)})\]
for the functor $K^Q:\mQ \to \W_E(\CC)$ to preserve the coproduct $Q^{(X)}$.

We call an object $Q\in \mQ$ \emph{self-small} if the functor $\Hom_\mQ(Q,-):\mQ \to \CC$
commutes with arbitrary set-indexed coproducts of copies of $Q$. Note that -- since the functor
$U^E :\W_E(\CC) \to \CC$ preserves and reflects all existing colimist (see, Subsection \ref{modules.}) --
this is equivalent to saying that the functor $\mQ(Q,-):\mQ \to \W_E(\CC)$  commutes with arbitrary set-indexed
coproducts of copies of $Q$. But coproducts in $\mQ$ are biproducts, and hence every object of $\mQ$ is self-small,
or equivalently, the functor $\mQ(Q,-)$ is $\T_E$-Galois for every object of $\mQ$.
\end{thm}

Combining the above discussion with Theorem \ref{equivalence.}, we obtain the following  necessary and sufficient
condition for the functor $K^Q$ from the diagram (\ref{set.case.}) to be an equivalence.

\begin{theorem}\label{equivalence.set.} With the data given in the adjoint triangle (\ref{set.case.}), suppose
that $\mQ$ admits arbitrary small copowers of $Q$. Then the functor $K^Q$
is an equivalence if and only if the functor $\mQ(Q,-):\mQ \to \emph{\Set}$ is monadic.
\end{theorem}

\begin{remark}\label{morita.r.} \em Since the category $\W_E(\CC)$  is naturally enriched over $\CC$, it follows from the 
very construction of the functor $K^Q$ that it is in fact (the underlying ordinary functor of) an $\CC$-enriched functor. 
Since the functor $\CC(\textsf{2},-): \CC \to \Set$ is conservative, an $\CC$-functor is equivalence if and only if its 
underlying ordinary functor is; so that the equivalence of categories in Theorem \ref{equivalence.set.}
can be promoted to an enriched equivalence of $\CC$-categories.
\end{remark}

Now we are ready to state and prove the key result of the paper, which provides equivalent conditions 
for a quantaloid to be equivalent to a module category over a quantale.

\begin{theorem} \label{sup.l.th.1} For a quantaloid $\mQ$,
the following conditions are equivalent:
\begin{itemize}
  \item [(i)] $\mQ$ is equivalent to a module category over a quantale.
  \item [(ii)] $\mQ$ is monadic over $\emph{\Set}$.
  \item [(iii)] $\mQ$ is Barr-exact and it has a regular projective generator $Q$
  and arbitrary small copowers of $Q$.
\end{itemize} When one of these equivalent conditions holds, then there exists an equivalence $\mQ \simeq
\emph{\W}_{E}(\CC)$, where $Q$ is as in \emph{(iii)} and $E=\emph{\Hom}_\mQ(Q,Q)$.
\end{theorem}
\begin{proof} Since any module category over a quantale is monadic over $\Set$ (see, \ref{modules.}), (i) implies (ii).
The equivalence (ii) $\Leftrightarrow$ (iii) is well known (e.g., \cite[Theorem 2.8]{MS}). Under (iii),
the functor $\mQ(Q,-):\mQ \to \Set$ is monadic and using Theorem \ref{equivalence.set.},  we see that 
$\mQ$ is equivalent to $\W_E(\CC)$, where $E=\Hom_\mQ(Q,Q)$. Thus, (iii) implies (i).
\end{proof}

Since categories which are monadic over $\Set$ are exactly, up to equivalence, (possibly non-finitary) varieties of
algebras (in the sense of universal algebra) (see, for instance, \cite{MaN}), Theorem \ref{sup.l.th.1} can be paraphrased as follows:

\begin{theorem} \label{paraphrased} The following are equivalent for a (finitary or infinitary) variety $\V$:
\begin{itemize}
  \item [(i)] $\V$ is enriched in $\CC$, i.e., $\V$  is a quantaloid.
  \item [(i)] $\V$ is equivalent to a module category over some quantale.
\end{itemize}
\end{theorem}

\begin{thm} \label{small.qt.}{\bf Functor categories are equivalent to module categories.}  \em
Given a small quantaloid $\mQ$, consider the functor \[Q=\oplus_{_{X \in \mQ}} \mQ(X, -): \mQ \to \CC.\]
Then $Q$ is regular projective generator for the category  $\CC\textsf{-Cat}(\mQ, \CC)_0$ of
$\CC$-functors and $\CC$-natural transformations from $\mQ$ to $\CC$ (see \cite{Mesab}).
Moreover, since the category $\CC$, being monadic over $\Set$, is Barr-exact, so also is
the category $\CC\textsf{-Cat}(\mQ, \CC)_0$. Applying Theorem \ref{sup.l.th.1} gives
that there is an equivalence $\CC\textsf{-Cat}(\mQ, \CC)_0\simeq \W_E(\CC)$, where
$E=\CC\textsf{-Cat}(\mQ, \CC)(Q,Q)$. Thus, we have proved the following, which is \cite[Theorem 1.1]{Mesab}:

\begin{theorem} \label{small.mor.}  For any small quantaloid $\mQ$, 
$\CC\textsf{-\emph{Cat}}(\mQ, \CC)_0$ of $\CC$-functors and $\CC$-natural transformations from $\mQ$ to $\CC$
is equivalent to a module category over a quantale.
\end{theorem}
\end{thm}

\
\
\

\bigskip

\section{Morita equivalent quantales}

In this section, we prove a Morita type theorem for quantales in terms of projective generators.

We say that two quantales  are \emph{Morita equivalent} if they have equivalent categories of right modules.
Essentially  the same proof as for ordinary rings yields:

\begin{theorem} Two commutative quantales are Morita equivalent if and only if they are isomorphic.
\end{theorem}

Since any module category over a quantale is Barr-exact, in the light of Remark \ref{remark.proj.}, we get 
from Theorem \ref{sup.l.th.1} the following theorem.

\begin{theorem} \label{morita.3.} Let $A$ be a quantale and $Q\in\emph{\W}_A(\CC)$ be a projective generator.
Then the quantales $A$ and $\emph{\Hom}_A(Q,Q)$ are Morita equivalent.
\end{theorem}

The following theorem gives a necessary and sufficient condition for two quantales to be Morita equivalent.

\begin{theorem} \label{morita.2.} Two quantales $A$ and $B$  are Morita equivalent if and only if
there exists a projective generator $Q \in \emph{\W}_A(\CC)$ such that $B\simeq \emph{\W}_A(\CC)(Q,Q)$ as quantales.
\end{theorem}
\begin{proof} Suppose first that two quantales $A$ and $B$ are Morita equivalent. Then there exists an
equivalence $\W_A \simeq \W_B$. From Remark \ref{remark.proj.} and then applying Theorem \ref{sup.l.th.1}, we deduce that there exists
a projective generator $Q \in \W_A(\CC)$ such that the functor \[K=\Hom_A(Q,-):\W_A(\CC)\to \W_B(\CC)\] is an equivalence
and the diagram
\[\xymatrix @R=38pt@C=45pt{
     \W_A(\CC)\ar[r]^-{K} \ar@/_1pc/[dr]|(0.4){\W_A(\CC)(Q,-)\dashv  Q^{(-)}}
      &\W_B (\CC) \ar[d]|-{B^{(-)} \dashv U_B}
      \\
      \text{ }
      & \Set }\]
is an adjoint triangle. Since $K$ is an equivalence, it follows from Theorem \ref{equivalence.} that the induced morphism
\[s_{K}=\CC(\textsf{2}, \gamma_{K}):\CC(\textsf{2},B^{(-)}) \to \CC(\textsf{2},\Hom_A(Q, Q^{(-)}))\]
of the monads on $\Set$ generated by the two adjunctions, is an isomorphism.
Then the map $\CC(\textsf{2},k):\CC(\textsf{2}, B) \to \CC(\textsf{2},\Hom_A(Q, Q))$, where $k:B \to \Hom_A(Q, Q)$
is the composite
\[B \xr{\simeq} \textsf{t} \times B \xr{(\gamma_{K})_{\textsf{t}}}\Hom_A(Q,Q^{(\textsf{t})})
\xr{\Hom_A(\CC)(Q,\,\simeq)}\Hom_A(Q,Q)\] is an isomorphism of ordinary monoids (e.g. \cite{FPN}).
Here, $\textsf{t}$ is a chosen terminal object in $\Set$, i.e., a set with only one element. It follows -- since the functor 
$\CC(\textsf{2},-):\CC \to \Set$ is conservative -- that $k:B \to \Hom_A(Q,Q)$ is an isomorphism of quantales.

For the converse, if $Q \in \W_A(\CC)$ is a projective generator, then there is an equivalence $\W_A(\CC) \simeq \W_{\Hom_A(Q,Q)}(\CC)$ 
by Theorem \ref{morita.3.}, and composing this with the evident equivalence $\W_{\Hom_A(Q,Q)}(\CC)\simeq \W_{B}(\CC)$
induced by the isomorphism $B\simeq \W_A(Q,Q)$ of quantales, gives us the desired equivalence $\W_A (\CC) \simeq \W_{B}(\CC)$.
\end{proof}

\
\
\
\

\section{Morita equivalence of quantales induced by idempotents}

\begin{thm} \label{sub.idemp.}{\bf Submodules defined by idempotents.}  \em
Let $A$ be a quantale, $M$ be a right $A$-module, $M'$ be a non-empty subset of $M$ and $A'$ a non-empty
subset of $A$. Put 
\[M'A=\{\bigvee_{m\in M' } m\cdot a_m \,\mid \,\{a_m\}_{m \in M'}\subseteq A\},\]
\[MA'=\{\bigvee_{a\in A'} m_a\cdot a \,\mid \,\{m_a\}_{a \in A'}\subseteq M\}.\]
Note that $M'A$ (resp. $MA'$) can be equally seen as the image of the composite 
\[M'\ot A \xr{i_{M'}\ot A} M \ot A \xr{\cdot} M \,\,\, \text{(resp.} \,\,\,M\ot A' \xr{M \ot i_{A'}} M \ot A \xr{\cdot} M),\] 
where $i_{M'}:M' \to M$ and $i_{A'}:A' \to A$ are the canonical inclusions. It is easy to see that $M'A$ is a submodule 
of the right $A$-module $M$ and is called the
\emph{submodule of $M$ generated by $M'$}. If, in addition, $M$ is an $(B,A)$-bimodule, $B$ being another quantale, 
then $MA'$ becomes a left $B$-module. When $M'=\{m\}$ (resp. $A'=\{a\}$) is one-point subset of $M$ (resp. $A$), then we write
$mA$ (resp. $Ma$) instead of $\{m\}A$ (resp. $M\{a\}$). Left and right side versions of these concepts are similarly defined.

Given two quantales $A, B$, a $(B, A)$-bimodule $M$ and an idempotent $e$ in $A$, one has an idempotent endomorphism of 
left $B$-modules $$-\cdot e: M \to M, m \longmapsto m \cdot e,$$ whose image (and hence the splitting object) is $Me$.
Since $Me$ is an equalizer of $-\cdot e$ and $\textsf{id}_M$, it follows that $Me=\{m\in M:m\cdot e=m\}$. Then 
the composite $M \xr{p_M} eM \xr{i_M} M$,  where $p_M$ is the map $m \to m \cdot e$, while $i_M$ is the canonical inclusion, 
is a splitting of the idempotent endomorphism $-\cdot e$. Hence in particular $Me$ is a left $B$-module. Symmetrically, for 
any  $(A,B)$-bimodule  $N$, the set $eN=\{e \cdot n: n\in N\}=\{n: n\in N:n=e\cdot n\}$ is a right $B$-module
and the idempotent endomorphism $N \xr{n \to n\cdot e}N$ of of right $B$-modules factors through $eN$. In particular,
the set $eA=\{ea: a\in A\}$ (resp. $Ae=\{a e: a\in A\}$) is a right (resp. left) $A$-module.  Observe that the
set $eAe=\{e a e: a\in A\}$, which can be seen as the sup-lattice $e(Ae)$ (or as $(eA)e$), is a quantale with unit $e$.

\end{thm}

\begin{proposition}\label{idempotent.} Let $A$ be a quantale and $e$ be an idempotent in $A$. For any right $A$-module $M$
there is a natural  isomorphism $\alpha_{M}:\emph{\Hom}_A(eA, M) \simeq Me$ in $\CC$. 
\end{proposition}
\begin{proof} Consider the composite \[q:\Hom_A(eA, M) \xr{\Hom_A(p_A, M)} \Hom_A(A, M)\simeq M,\] where $p_A:A \to eA$
is the first morphism in the factorization $A\xr{p_A}eA \xr{i_A}A$ of the idempotent morphism $e\cdot -:A \to A$. Since $p_A$ is
an epimorphism, $q$ is injective. We claim that the image $\textsf{Im}(q)$ of $q$ is $Me$. 
Indeed, since for any $f\in \Hom_A(eA, M)$, $\Hom_A(p_A, M)(f)=fp_A$, it follows that $q(f)=(fp_A)(\textsf{1})=f(p_A(\textsf{1}))=f(e)$ and then 
\[q(f)\cdot e=f(e)\cdot e=f (ee)=f(e)=q(f). \] Hence, $\textsf{Im}(q)\subseteq Me$. Conversely,
if $m\in Me$, then $m\cdot e=m$ and if $f_m:eA \to M$ is given by $f_m(ea)=m\cdot a$, then $f_m$ is defined correctly (for, if $ea=ea'$,
then $m\cdot a=(m \cdot e)\cdot a=m \cdot (ea)=m \cdot (ea')=(m \cdot e)\cdot a'=m\cdot a'$) and is a morphism of right 
$A$-modules. Moreover, $q(f_m)=f_m(e)=m \cdot e=m$ and hence $Me \subseteq \textsf{Im}(q)$. Consequently, $q$ induces an isomorphism 
$\Hom_A(eN, M) \simeq \Hom_A(N, M)e$.
\end{proof}

\begin{remark}\label{eA} \em It is straightforward to verify that the isomorphism \[\alpha_A:\Hom_A(eA, eA) \simeq eAe\] is in fact 
 an isomorphism of quantales.
\end{remark}

Following the terminology of ring theory, we shall call an idempotent $e$ in a quantale $A$ \emph{full} if $AeA=A$.
Here $AeA=(Ae)A (=A(eA))$.

\begin{proposition}\label{generator.} Let $A$ be a quantale and $e$ be an idempotent in $A$. The right $A$-module $eA$ is a
generator for the category $\emph{\W}_A(\CC)$ if and only if the idempotent $e$ is full.
\end{proposition}
\begin{proof} Note first that by Remark \ref{remark.proj.}, the right $A$-module $eA$ is generator for $\W_A(\CC)$ if and only if
it is a strong generator, and hence if and only if the functor $\W_A(eA,-):\W_A(\CC)\to \Set$ (or, equivalently, the functor 
$\Hom_A(eA,-):\W_A(\CC)\to \CC$, since $U^A:\W_A(\CC) \to \CC$ is conservative) reflects isomorphisms. But since the latter functor 
admits the functor $-\ot eA :\CC \to \W_A(\CC)$ as a left adjoint, $\Hom_A(eA,-):\W_A(\CC)\to \CC$ reflects isomorphism if and only if each component
$\sigma_M$, $M \in \W_A(\CC)$, of the counut $\sigma$ of the adjunction $-\ot eA \dashv \Hom_A(eA,-)$, which is the
evaluation map $\Hom_A(eA, M)\ot eA \to M$, is an epimorphism (see \ref{generator.def.}). This is equally to require that  the composite 
\[\kappa_M:Me \ot eA \xr{\alpha^{-1}_{eA,M}\ot eA} \Hom_A(eA, M)\ot eA \xr{\sigma_M}M\] is an epimorphism.
Therefore, $eA$ is a generator for $\W_A(\CC)$ if and only if for each $M \in \W_A(\CC)$, the map $\kappa_M$, which is easily seen to 
take $\bigvee_{i \in I}(m_i \cdot e\ot e a_i)$ to $\bigvee_{i \in I}m_i \cdot (e a_i)$, is an epimorphism.

Suppose now that $eA$ is a generator for $\W_A (\CC)$. Then each $\kappa_M$, and in particular $\kappa_A:Ae \ot eA \to A$,
is an epimorphism. It follows --  since the image of $\kappa_A:Ae \ot eA \to A$  is $AeA$ -- that $AeA=A$, i.e., 
that $e$ is a full idempotent. Conversely, if $e$ is a full idempotent, then $AeA=A$ and therefore there are elements
$(a_i\in A)_{i \in I}$ and $(a'_i\in A)_{i \in I}$ such that $\bigvee_{i \in I}a_iea_i'=\textsf{1}$. Then we have for any 
$M \in \W_A(\CC)$ and any $m \in M$:  
\[\begin{split} \kappa_M(\bigvee_{i \in I}(m\cdot (a_ie) \ot ea_i'))&=\bigvee_{i \in I}(\kappa_M)(m\cdot (a_ie) \ot ea_i')\\
&=\bigvee_{i \in I}(m\cdot (a_iea_i'))\stackrel{(\ref{module})}=m\cdot \bigvee_{i \in I}(a_iea_i'))\\
& =m \cdot \textsf{1}\stackrel{(\ref{module})}=m,
\end{split}\] where the first equality holds because $\kappa_M$ is a morphism in $\CC$, and the second equation holds by 
definition of $\kappa_M$. It follows that $\kappa_M$ is an epimorphism for all right $A$-module $M$, implying
that $eA$ is a generator for $\W_A(\CC)$.
\end{proof}

\begin{theorem}\label{eAe} Let $A$ be a quantale and $e$ be a full idempotent in $A$. Then
the quantales $A$ and $eAe$ are Morita equivalent.
\end{theorem}
\begin{proof} Since the right $A$-module $eA$ is a splitting object of the idempotent morphism $e\cdot -$, $eA$ is a retract of
$A$ in $\W_A(\CC)$ and thus is projective (e.g.,  \cite[Proposition 1, page 11]{JT}). Moreover, since $e$ is assumed
to be a full idempotent, $eA$ is generator for $\W_A(\CC)$, by Proposition \ref{generator.}.  It  then follows form 
Theorem \ref{morita.3.} that $A$ is Morita equivalent to the quantale $\Hom_A(eA,eA)$. But, by Proposition \ref{idempotent.},
the latter quantale is isomorphic to $eAe$. As a consequence, the quantales $A$ and $eAe$ are Morita equivalent.
\end{proof}

\
\
\

\section{Morita equivalence in terms of matrices}

In this section, we prove a Morita type theorem for quantales in terms of matrices.

\begin{thm}\label{matrices.} {\bf Matrices with values in a quantale.} \em Let $A$ be an arbitrary quantale. 
If $X$ and $Y$ are sets, an \emph{A-matrix} $\textsf{a}$  of size $X\times Y$ is a family $(\textsf{a}(x, y))_{(x,y)
\in X\times Y}$ of elements of $A$. Note that $A$-matrices of size  $X\times Y$ might be equally seen as  maps
$X\times Y \to A$. If $\textsf{a}$ and $\textsf{b}$ are two $A$-matrices with size $X\times Y$ and
$Y\times Z$, then their product $\textsf{a} \star \textsf{b}$ has size  $X\times Z$ and is defined by 
\[(\textsf{a} \star \textsf{b})(x,z)=\bigvee_{y \in X} \textsf{a}(x,y)\textsf{b}(y,z).\] For any set $X$, we write $\textsf{i}_X$ for the
matrix of size $X\times X$ given by $$\textsf{i}_X(x,x')=\begin{cases}
1, ~x=x'\\ 0, ~x\neq x'
\end{cases}$$ $\textsf{i}_X$ serves as the identity for the matrix multiplication. We write $\Mat_{X\times Y}(A)$ for the set
of all $A$-matrices of size $X\times Y$. $\Mat_{X\times Y}(A)$ is a sup-lattice with operations formed point-wise from those in $A$.
When $X=Y$, we shall just write $\Mat_{X}(A)$ instead of $\Mat_{X\times X}(A)$. It is routine to check that
the triple $(\Mat_{X}(A), \star, \textsf{i}_X)$ is a quantale.

For any set $X$, the right free $A$-module on $X$ has the form $A^{(X)}$ (see, for instance, \cite{JT}), which 
is isomorphic to $A^{X}$, since coproducts and products are isomorphic in $\W_A(\CC)$. Then identifying $A^{X}$ with the 
set of mappings $X\times \textsf{t}\to A$ (recall that $\textsf{t}$ is a terminal object in $\Set$), we have for any 
$\textsf{a}\in \Mat_{X}(A)$ a morphism $\textsf{a}\cdot -:A^{(X)} \to A^{(X)}$ of right $A$-modules
given by  \[(\textsf{a} \cdot f)_{x}= \bigvee_{x'\in X }\textsf{a}(x,x')f(x'),\] and the assignment 
$\textsf{a} \longmapsto \textsf{a} \cdot -$
yields an isomorphism  \[j_X:\Mat_{X}(A) \to \Hom_A(A^{(X)}, A^{(X)})\]
of quantales. Note that $j_X$ is the inverse of the isomorphism 
\[\Hom_A(A^{(X)}, A^{(X)})\simeq A^{(X\times X)}\simeq A^{X\times X}=\Mat_{X}(A)\] of sup-lattices.

\begin{theorem}\label{free.} Let $A$ be a quantale. For any set $X$, $A$ and $\emph{\Mat}_{X}(A)$ are Morita equivalent 
quantales.
\end{theorem}
\begin{proof} The result follows from Theorem \ref{morita.2.} by noting that $A^{(X)}$ is a free projective 
generator for $\W_A$ and that $j_X$ is an isomorphism of quantales.
\end{proof}

For an idempotent matrix  $\textsf{a} \in \Mat_{X}(A)$, we shall write $\textsf{a} A^{(X)}$ for the splitting object of the idempotent morphism
$\textsf{a} \cdot -: A^{(X)} \to A^{(X)}$ of right $A$-modules.

\begin{lemma}\label{iso.} Let $A$ be a quantale. For any set $X$ and any idempotent matrix $\emph{\textsf{a}} \in \emph{\Mat}_{X}(A)$,
there are isomorphisms  \[\emph{\Hom}_A( \emph{\textsf{a}} A^{(X)}, A^{(X)}) \simeq \emph{\Mat}_{X}(A) \emph{\textsf{a}}\] and
\[\emph{\Hom}_A( \emph{\textsf{a}} A^{(X)}, \emph{\textsf{a}} A^{(X)}) \simeq \emph{\textsf{a}} \,\emph{\Mat}_{X}(A) \emph{\textsf{a}}\]
of sup-lattices.
\end{lemma}
\begin{proof} As any functor preserves idempotents and their splittings, by applying the functor $\Hom_A(-, A^{(X)}):\W_A(\CC)\to \CC^{op}$
to the factorization $A^{(X)} \xr{p} \A A^{(X)} \xr{i} A^{(X)}$ of the idempotent morphism $\textsf{a}\cdot- :A^{(X)} \to A^{(X)}$
of right $A$-modules, we obtain an idempotent morphism \[\Hom_A(\textsf{a}\cdot-, A^{(X)}) :  \Hom_A(A^{(X)}, A^{(X)})\to  \Hom_A(A^{(X)}, A^{(X)})\]
in $\CC$, split by  $\Hom_A(i, A^{(X)})$ and  $\Hom_A(p, A^{(X)})$. An easy straightforward calculation shows
that the rectangle in the diagram
\[\xymatrix @R=28pt@C=29pt{ \Mat_{X}(A) \ar[d]_{j_X}\ar[rr]^-{-\star \textsf{a}}&&\Mat_{X}(A) \ar[d]^{j_X}\\
\Hom_A(A^{(X)}, A^{(X)}) \ar[rd]_-{\Hom_A(i, A^{(X)})} \ar[rr]^-{\Hom_A(\textsf{a} \,\cdot -, \,\,A^{(X)})}&& \Hom_A(A^{(X)}, A^{(X)}) \\
& \Hom_A( \textsf{a} A^{(X)}, A^{(X)}) \ar[ru]_-{\Hom_A(p, A^{(X)})}&}\] commutes and therefore the whole diagram also commutes.
It follows -- since $j_X$ is an isomorphism -- that $\Hom_A( \textsf{a} A^{(X)}, A^{(X)})$ is a splitting object of the idempotent morphism
$-\star \textsf{a}: \Mat_{X}(A) \to \Mat_{X}(A)$. But since a splitting object is  unique up to isomorphism, and since a splitting object of
$-\star \textsf{a}$ is $\Mat_{X}(A)\textsf{a}$, it follows that $\Hom_A( \textsf{a} A^{(X)}, A^{(X)})$ and  $\Mat_{X}(A)\textsf{a}$ are isomorphic.

The second isomorphism can be established by applying similar arguments as above to the commutative diagram
\[\xymatrix @R=28pt@C=29pt{ \Mat_{X}(A) \ar[d]_{j_X}\ar[rr]^-{\textsf{a} \,\star -\star\,\textsf{a}}&&\Mat_{X}(A) \ar[d]^{j_X}\\
\Hom_A(A^{(X)}, A^{(X)}) \ar[rd]_-{\Hom_A(i, q)} \ar[rr]^-{\Hom_A(\textsf{a} \,\cdot -, \,\textsf{a}\,\cdot -)}&& \Hom_A(A^{(X)}, A^{(X)}) \\
& \Hom_A( \textsf{a} A^{(X)}, \textsf{a} A^{(X)}) \ar[ru]_-{\Hom_A(q, i)}&}\]
\end{proof}

\begin{theorem}\label{m.full.idemp.} Two quantales $A$ and $B$ are Morita equivalent if and only if there exist a set $X$ and
a full idempotent $\emph{\textsf{a}} \in \emph{\Mat}_{X}(A)$ such that $B\simeq \emph{\textsf{a}}\, \emph{\Mat}_{X}(A)\, 
\emph{\textsf{a}}$ as quantales.
\end{theorem}
\begin{proof} If $\textsf{a} \in \Mat_{X}(A)$ is a full idempotent such that the quantales $B$ and $\textsf{a}\,\Mat_{X}(A)\,\textsf{a}$ 
are isomorphic, then $\Mat_{X}(A)$ and $\textsf{a}\,\Mat_{X}(A)\,\textsf{a}$ are Morita equivalent quantales by Theorem \ref{eAe}, 
while $A$ and $\Mat_{X}(A)$ are equivalent by Theorem \ref{free.}. Therefore, that the quantales $A$ and $\textsf{a}\,\Mat_{X}(A)\,\textsf{a}$ 
are Morita equivalent.

Conversely, suppose that $A$ and $B$ are Morita equivalent quantales. Then, by Theorem \ref{morita.2.}, there exists a projective
generator $Q \in \W_A(\CC)$ such that $\Hom_A(Q,Q)\simeq B$. Since $Q$ is projective in $\W_A(\CC)$, there exists a set $X$ such that
$Q$ is retract of the right $A$-modules $A^{(X)}$ (see, for example, \cite[Proposition 1, page 11]{JT}). In other words,
there exists an idempotent $e\in \Hom_A(A^{(X)}, A^{(X)})$ such that $Q$ is the splitting object of $e$. Put $\textsf{a}=j^{-1}_X(e)$.
Then $\textsf{a} \in \Mat_{X}(A)$ is an idempotent matrix and $e$ can be identified with the idempotent endomorphism 
$\textsf{a} \cdot -:A^{(X)}\to A^{(X)}$. It follows that $Q$ is isomorphic to the splitting object of $\textsf{a}\cdot -$, which is 
$\textsf{a} A^{(X)}$. Since  $A$ and $B$ are Morita equivalent quantales by hypothesis, $B\simeq \Hom_A(Q,Q)$. But $\Hom_A(Q,Q)\simeq 
\Hom_A(\textsf{a} A^{(X)},\textsf{a} A^{(X)})$ and since $\Hom_A(\textsf{a} A^{(X)}, \textsf{a}A^X)\simeq \textsf{a}\,\Mat_{X}(A)\textsf{a}$ 
by Lemma \ref{iso.}, it follows that $B\simeq \textsf{a}\,\Mat_{X}(A)\textsf{a}$; so it remains to show that $\textsf{a} \in \Mat_{X}(A)$ 
is a full idempotent. Since $A^{(X)}$ is a retract of $A^{(X\times X)}$ in $\W_A(\CC)$, $\textsf{a} A^{(X)}$ is a retract of 
$\textsf{a} A^{(X\times X)}$ in $\W_A(\CC)$. As $\textsf{a} A^{(X)}$ is a generator for $\W_A(\CC)$, it follows that $\textsf{a} A^{(X\times X)}$ 
is also a generator for $\W_A(\CC)$ (see \ref{generator.def.}) and hence the functor
$\Hom_A(\textsf{a}A^{(X\times X)},- ):\W_A(\CC) \to \CC$ is conservative. Then each component of the counit of the adjunction
$ \Hom_A(\textsf{a} A^{(X\times X)},- ) \dashv -\ot \textsf{a} A^{(X\times X)}$ is epimorphic. In particular, the $A^{(X\times X)}$-component
of the counit, which is the evaluation map \[\Hom_A(\textsf{a} A^{(X\times X)},A^{(X\times X)})\ot \textsf{a} A^{(X\times X)}  \to A^{(X\times X)},\]
is an epimorphism. But since $\Hom_A(\textsf{a}A^{(X\times X)}, \,\,A^{(X\times X)})\simeq A^{(X\times X)} \textsf{a}$ by Lemma \ref{iso.}, the $A^{(X\times X)}$-component
is an epimorphism if and only if the morphism \[A^{(X\times X)} \textsf{a} \ot \textsf{a} A^{(X\times X)}  \to A^{(X\times X)},\] or, equivalently, the
morphism \[\Mat_{X}(A)\textsf{a} \ot \textsf{a}\Mat_{X}(A) \to \Mat_{X}(A)\]
given by the matrix multiplication, is an epimorphism. And this is the case if and only if the image of the last morphism,
which is $\Mat_{X}(A)\,\textsf{a} \,\Mat_{X}(A)$, is isomorphic to $\Mat_{X}(A)$, i.e., if and only if $\textsf{a}$ is a full idempotent. 
This completes the proof of the theorem.
\end{proof}

\end{thm}

\section{Some applications: Sup-lattices in Grothendieck toposes}

In this section, we deal with internal sup-lattices in Grothendieck toposes. It is well known that sup-lattices can be defined 
in any elementary topos (e.g., \cite{JT}).

\begin{thm} \label{topos.}{\bf Sup-lattices in elementary toposes.}  \em Let $\E$ be an elementary topos (\cite{JP}) with 
the subobject classifier $\Omega$, and let $\CC(\E)$ denote the category of internal sup-lattices (i.e. internally complete
lattices and sup-preserving morphisms) in $\E$. It is well known (see, \cite{JT}) that $\CC(\E)$ is monadic over $\E$.  The 
covariant power-set functor $(X \to \Omega^X, \,  f \to \exists_f)$ (e.g., \cite{JP}) is a monad on $\E$ whose unit is the 
inclusions of singletons $ \{\}_X\colon X \to \Omega^X$, and whose multiplication $\Omega^{\Omega^X} \to \Omega^X$
the internal union, respectively. This monad is denoted by $\mP$. An object $X \in \E$ is an $\mP$-algebra precisely if $X$ 
is an internal sup-lattice in $\E$, and thus $\CC(\E)$, being (isomorphic to) the category of $\mP$-algebras, is monadic over 
$\E$. The free $\mP$-algebra functor $F^\mP :\E \to \CC(\E)$ takes an object $X \in \E$ to $\mP(X)$ and therefore, for all
$X \in \E$ and $S \in \CC(\E)$, we have a bijection 
\begin{equation}\label{topos.eq.}\CC(\E) (\mP(X), S)\simeq \E(X, S).\end{equation}
Since the forgetful functor $ \CC(\E) \to \E$ takes epimorphisms to split epimorphisms \cite{JT}, Remark \ref{remark.proj.}
applies also to internal sup-lattices. In particular, generators are regular (and hence also strong) generators in $\CC(\E)$.
Moreover, the free $\mP$-algebras are projective in $\CC(\E)$ (\cite{V}). Observe next that since $\CC(\E)$ is monadic on $\E$, 
all small limits (and hence all small colimits) exist in $\CC(\E)$ provided they exist in $\E$. Therefore, if $\E$ is a Grothendieck 
topos, $\CC(\E)$  is complete and cocomplete. Finally, recall that $\CC(\E)$ is Barr-exact (e.g., \cite{V}).
\end{thm} 

\begin{thm} \label{relations.}{\bf Relations in elementary toposes.}  \em
Recall that for objects $X$ and $Y$ in any category $\E$ with finite products, a \emph{relation} from $X$ to $Y$ is a subobject
of $X \times Y$. It is well known that, in any Barr-exact category relations can be composed using pullbacks and image factorizations.
The identity on $X$ for this composition is given by the diagonal morphism $X \to X \times X$. Now if $\mathcal E$ is a Grothendieck
topos, then for any object $X \in \E$, the set $\textsf{Sub}(X)$ of subobjects of $X$ is a complete lattice, i.e., a sup-lattice:
the supremum of subobjects $X_i \rightarrowtail X$ in $\E$ can be calculated by taking the image of the induced morphism $\oplus X_i 
\to X$ (see, for example, \cite{MclM}). Consequently, the set $\textsf{Sub}(X \times X)$ of relations on $X$  with its subobject 
ordering and relational composition is a quantale.

We record the following lemma, which we shall need. Recall that a \emph{bound} of a Grothendieck topos $\E$ is an object whose subobjects 
form a generating family for $\E$.

\begin{lemma}\label{bound.g.} Let $B$ be a bound of a Grothendieck topos $\E$. Then $\mP(B)=\Omega^B$
is a generator for $\CC(\E)$.
\end{lemma}
\begin{proof} Note first that since $B$ is a bound of $\E$, its subobjects $\{k_i:B_i \to B, i \in I\}$ generate $\E$. Since the forgetful 
functor  $\CC(\E) \to \E$ is evidently faithful and the functor $F^{\mathcal E}=\Omega^{-}$ is a left adjoint to it, we conclude (for example, 
from \cite[Chapter 5, Proposition 1.5]{BMit}) that  the  family  $\{\Omega^{B_i}, i \in I\}$ is  generating for $\CC(\mathcal E)$. Next, for 
each $i \in I$, the (mono)morphism $k_i:B_i \to B$, like any other morphism in $\E$, yields morphisms $\mP(k_i)=\Omega^{k_i}:\mP(B) \to \mP(B_i)$ 
and $\forall_{k_i}:\mP(B_i) \to \mP(B)$ such that $\mP(k_i)$ is an internal right adjoint of $\exists_{k_i}$ (e.g., \cite[Theorem IV.9.2]{MclM}) 
and  $\forall_{k_i}$ is an internal right adjoint of $\mP(k_i)$ (e.g., \cite[Propositions IV.9.4]{MclM}). Since internal left adjoints are 
supremum-preserving morphisms (see, for example, \cite{KLM}), it follows that $\mP(k_i)$ is a morphism of internal sup-lattices. But since $k_i:B_i 
\to B$ is a monomorphism, $\mP(k_i) \circ \exists_{k_i}=\textsf{id}_{\Omega^{B_i}}$ by \cite[Corollary 1.33]{JP}, and hence each $\Omega^{B_i}$ is 
a retract  of $\Omega^{B}$ in $\CC(\E)$. Then $\Omega^{B}$ is a generator for $\CC(\E)$. For if $h:\Omega^{B_i}\to X $ separates two different 
morphisms  $f,g: X\rightrightarrows Y$, then the compopsite $h \circ \mP(k_i)$ also separates $f$ and $g$, since $\mP(k_i)$ is a (split) epimorphism.
\end{proof}

Now we are ready to prove the following result, which gives a characterization of internal sup-lattices in a Grothendieck topos in terms of external 
modules over an external quantale.

\begin{theorem}\label{pitts.} Let $B$ a bound of a Grothendieck topos $\mathcal E$. Then there is an equivalence of categories \[\CC(\mathcal E) \simeq \emph{\W}_{\emph{\textsf{Sub}}(B \times B)}(\CC).\]
\end{theorem}
\begin{proof} As mentioned above, $\CC(\E)$ is Barr-exact. Since $\mathcal E$ is a Grothendieck topos, $\CC (\mathcal E)$ is enriched in $\CC$ (see 
\cite{Pit}). Next,  $\mP(B)=\Omega^B$, being a free internal semi-lattice, is projective in $\CC(\E)$; moreover, it is a generator for $\CC(\E)$ by 
Lemma \ref{bound.g.}. Applying Theorem \ref{sup.l.th.1}, we conclude that the functor \[\CC(\E)(\mP(B),-):\CC(\E)\to \W_{\CC(\mathcal E)(\mathcal P(B), 
\mathcal P(B))}(\CC)\] is an equivalence of categories. The desired result now follows from the fact that there is an isomorphism $\CC(\mathcal E)
(\mathcal P(B), \mathcal P(B))\simeq \textsf{Sub}(B\times B)$ of quantales, as can be seen from the following sequence of natural isomorphisms of 
sup-lattices: \[\CC(\mathcal E)(\mathcal P(B), \mathcal P(B))\stackrel{(\ref{topos.eq.})}\simeq \mathcal E(B, \mathcal P(B))=\mathcal E(B, \Omega^B)\simeq
\mathcal E(B\times B, \Omega)\simeq\textsf{Sub}(B\times B).\]  Here the second isomorphism comes from the exponential adjunction and  the third from the definition of $\Omega$.
\end{proof}
\end{thm}

\begin{remark} \em This result should be compared with \cite[Theorem 5.2]{Pit} which gives a different description of the category of internal sup-lattices
in Grothendieck toposes.
\end{remark}

If a bound $B$ is a subquotient  of an object $B'$ (i.e. if $B$ is a subobject of a quotient of $B'$), 
then $B'$ is also a bound. Hence Grothendieck toposes have infinitely many bounds. So we have infinitely many choices for $B$
in Theorem \ref{pitts.}. Nevertheless, as the following result asserts, their corresponding quantales $\textsf{Sub}(B \times B)$ 
are unique up to Morita equivalence.

\begin{theorem} If $B$ and $B'$ both are bounds of a Grothendieck topos $\mathcal E$, then the quantales $\emph{\textsf{Sub}}(B\times B)$
and $\emph{\textsf{Sub}}(B'\times B')$ are Morita equivalent. 
\end{theorem}
\begin{proof} Trivially follows from Theorem \ref{pitts.}
  \end{proof}

\begin{thm} \label{locales.}{\bf Sup-lattices in localic toposes.}  \em
Recall from \cite{JT} that a locale is a commutative idempotent quantale $L$ such that $l \leq 1$ for all $l \in L$.
For example, for any object $X$ of a Grothendieck topos $\E$, the set $\textsf{Sub}(X)$ is a locale (\cite[Proposition III. 8.1]{MclM})).
Since any locale $L$ is a partially ordered set, it can be seen as a category $\textsf{L}$ whose objects are the elements of $L$, and 
there is a morphism $l \to l'$ in $\textsf{L}$ if and only if $l \leq l'$. Define a Grothendieck topology $J$ on $\textsf{L}$ by 
assigning to each element  $l$ of $L$ (i.e. object of $\textsf{L}$) the collection $J_l=\{(l_i \leq l)_{i \in I}\}$  such that 
$\bigvee_{i \in I}l_i=l$. We shall write $\textsf{Sh}(\textsf{L},J)$ (or just $\textsf{Sh}(\textsf{L})$) for the category of sheaves 
on the site $(\textsf{L},J)$. Note that this topology is canonical in the sense that it is the largest Grothendieck topology on $\textsf{L}$ in which
all representable functors are sheaves. Since the locale $\textsf{Sub}(\textsf{T})$ of subsheaves of the terminal sheaf $\textsf{T}$ 
in $\textsf{Sh}(\textsf{L})$ is isomorphic to the locale $L$ (see, for example, the proof of \cite[Theorem IX. 5.1]{MclM}) and since 
$\textsf{T} \times \textsf{T}\simeq \textsf{T}$, applying Theorem \ref{pitts.}, we obtain the following result of Joyal and Tierney 
\cite[Chapter VI, Proposition 3.1]{JT}:
\end{thm}

\begin{theorem}\label{JT} For any locale $L$, there is a natural
equivalence $$\CC(\emph{\textsf{Sh}}(\emph{\textsf{L}})) \simeq \emph{\W}_L(\CC).$$
\end{theorem}

\end{document}